\documentclass[11pt]{amsart}
   \usepackage{booktabs}                                 
\usepackage[english]{babel}
\usepackage{latexsym}
 \usepackage[utf8]{inputenc}
\usepackage{babel}
\usepackage{fancyhdr}
\usepackage{longtable}
\usepackage{mathrsfs}
\usepackage{geometry}
\usepackage{textcomp}
\usepackage{caption}
\usepackage{lmodern}
\usepackage{mathrsfs}
 \usepackage[T1]{fontenc}
\usepackage[all,cmtip]{xy}
\usepackage{epsfig}
 \usepackage{amsmath,amsfonts,amssymb,amsthm}
\usepackage{graphics}
\usepackage{graphicx}
\usepackage{verbatim}
\usepackage{dsfont}

\usepackage{pdflscape}
\usepackage{longtable}
\usepackage{booktabs}
\usepackage[usenames, dvipsnames]{color}

\usepackage[colorlinks=true, pdfstartview=FitV, linkcolor=blue, 
            citecolor=blue, urlcolor=blue]{hyperref}

\newcommand{\C}{\mathbb{C}}

\newcommand{\Q}{\mathbb{Q}}
\newcommand{\Z}{\mathbb{Z}}

\newcommand{\PP}{\mathbb{P}}

\newcommand{\tW}{\widetilde{W}}
\newcommand{\tS}{\widetilde{S}}
\newcommand{\tV}{\widetilde{V}}

\newcommand{\of}{\mathcal{O}}

\newcommand{\W}{\bigwedge}

\DeclareMathOperator{\Ann}{Ann}

\DeclareMathOperator{\Pf}{Pf}

\DeclareMathOperator{\Ker}{Ker}

\DeclareMathOperator{\ddim}{dim}
\DeclareMathOperator{\Gr}{Gr}

\DeclareMathOperator{\HP}{HP}

\newtheorem{thm}{Theorem}[section]

\newtheorem{corollary}[thm]{Corollary}

\newtheorem{lemma}[thm]{Lemma}

\newtheorem{proposition}[thm]{Proposition}

\theoremstyle{definition}

\newtheorem{thm*}{Theorem}
\newtheorem{proposition*}{Proposition}

\makeatletter    
\def\l@subsection{\@tocline{1}{0,2pt}{2pc}{8mm}{\ \ }} 
\makeatletter    
\def\l@section{\@tocline{1}{0,2pt}{2pc}{8mm}{\ \ }}

\author{Enrico Fatighenti }
\address{Dipartimento di Matematica e Fisica\\
  Università Roma 3\\
  L.go S. L. Murialdo, 1, 00146 Roma, Italia}
\email[E.~Fatighenti]{efatighenti@mat.uniroma3.it}
 \title{Surfaces of general type with $p_g=1$, $q=0$, $K^2=6$ and grassmannians}

\begin{document}
\maketitle

\begin{abstract}
We construct examples of surfaces of general type with $p_g=1$, $q=0$ and $K^2=6$. We use as key varieties Fano fourfolds and Calabi-Yau threefolds that are zero section of some special homogeneous vector bundle on Grassmannians. We link as well our construction to a classical Campedelli surface, using the Pfaffian-Grassmannian correspondence.
\end{abstract}

\section{Introduction}
The classification of surfaces of general type is one of the most active areas of algebraic geometry. Many examples are known, but a detailed classification is still lacking (and maybe even impossible to accomplish), and several hard problems are still open. 
 To each minimal surface of general type we will associate a triple of numerical invariants, $(p_g, q, K^2_S)$, where $p_g:=h^0(S, K_S)$ and $q:=h^1(S, \of_S)$. These indeed determine all other classical numerical invariants, such as $e_{\textrm{top}}(S)=12\chi(\of_S)-K^2_S$ and $P_{m}(S):=h^0(S, mK_S)=\chi(\of_S) + {m \choose 2} K^2_S$. For a recent survey on the surfaces of general type we refer to \cite{bauer2006complex}.\\
Two very simple ways to produce surfaces of general type are complete intersections of sufficiently high multi-degree or product of curves with $g \ge 2$. These produces surfaces with either large $p_g$ or $q$. This is a particular manifestation of  more general phenomenon: producing examples of surfaces of general type with low $p_g$ and $q$ is indeed quite difficult, and a complete classification is beyond the current level of research. A useful tool to produce such examples consist in the identification of families of surfaces of general type whose general member $S$ is invariant with respect to a finite group $H$, and taking the quotient $S/H$. The archetypal example is due to Godeaux, and is realised as the quotient $Y_5/\Z/5$, where $Y_5 \subset \PP^3$ is a quintic surface in $\PP^3$ on which the group $\Z/5$ acts freely. Surfaces with $p_g=q=0, K^2_S=1$ are therefore called \emph{(numerical) Godeaux surfaces}. Similarly one can construct explicit examples of a surface with $p_g=q=0, K^2_S=2$ as quotient for a $\Z/7$ action. Indeed surfaces with these prescribed invariants are called \emph{(numerical) Campedelli surfaces}. We will recall later in full details this construction. 
Much less is known on surfaces with $p_g=1, \ q=0$. One important class of examples is given by \emph{Todorov surfaces}, i.e. minimal smooth surfaces $S$ of general type with $q=0$ and $p_g=1$ having an involution $i$ such that $S/i$ is birational to a K3 surface and such that the bicanonical map of $S$ is composed with $i$. They constitute a counterexample to the Torelli problem, \cite{todorov}. Some more (simply connected) examples are obtained by Park, Park and Shin in \cite{park} using the technique of $\Q$-Gorenstein smoothing. This paper is devoted to find some new examples of surfaces of general type with $p_g=1, \ q=0, \ K^2=6$.  Our examples are neither of Todorov nor Park-Park-Shin type. Indeed for Todorov surfaces we have a complete list of 11 non empty irreducible families,
\cite[pg 335]{morrison}, with our surface being in none of them. Similarly the surfaces in the Park-Park-Shin list are simply connected. To the best of our knowledge the surfaces we construct are new, as the construction is. Our hope is to use similar methods in future to construct further examples of surfaces.
\\ We will explain some interesting connections (via the Pfaffian-Calabi Yau correspondence) to a classical construction of a Campedelli surface as well.
Finding examples of invariant subfamilies in $\PP^n$  with the right numerology can be difficult. On the other hand the lists  of Fano fourfolds and Calabi-Yau threefolds of K\"uchle and Inoue-Ito-Miura, cf. \cite{kuchle}, \cite{inoue2016complete} provides an excellent source of potential key varieties.\\
The general strategy goes as follows: first we look for an explicit subfamily invariant with respect to a finite group of the automorphism group. Provided that the fixed locus of the action is not too big, we then take the quotient. This strategy can be potentially applied to many of the examples in the above mentioned lists. Here we choose a particularly nice case. The starting point is the analysis of two Fano fourfolds of index 1 in Grassmannians Gr(2,6) and Gr(2,7). These Fanos are constructed as  (resp.) zero locus of a general section of the twisted quotient bundle and 6-codimensional general linear section. They appears in K\"uchle list as (b3) and (b7) and were shown to be projectively equivalent in a recent work of Manivel (\cite{manivel}). From these one can get to the level of surfaces by simply picking two furher hyperplane sections. These are surfaces of general type with $p_g=13, K^2=42$.\\ We explicitely show how to construct an action of the dihedral group $D_7$ of order 14 on these Fanos, and how to pick $D_7$-invariant linear subspaces such that the corresponding surfaces are smooth, with a free $\Z/7 \triangleleft D_7$ action. This in turn will allow us to produce new examples of surfaces of general type and Calabi-Yau threefolds. We can summarise our main results as 
\begin{thm}[Thm. \ref{cyquotient}, Cor. \ref{surface}, Prop. \ref{g36}] We construct examples of surfaces of general type with $p_g=1,\ q=0,\ K^2=6$ as
\begin{itemize}
\item $\Z/7$ quotient of a 8-codimensional invariant linear section of the Grassmannian $\Gr(2,7)$, or equivalently, of the zero set of an invariant section of the vector bundle $\mathcal{E}= \mathcal{Q}(1) \oplus \of(1)^{\oplus 2}$ on $\Gr(2,6)$
\item $\Z/7$ quotient of a 7-codimensional invariant linear section of the Grassmannian $\Gr(3,6)$.
\end{itemize}
\end{thm}
One of the question we want to answer in a future work is if the link between the two surfaces in above theorem (in particular whether they belong to the same family).\\
In the end we explain how to link our construction to another famous surface of general type the $\Z/7$ Campedelli-Reid surface $\tV$ constructed in \cite{miles}. Indeed we extend our surface $\tS$ to a Calabi-Yau threefold $\tW$. Using the famous Pfaffian-Calabi Yau correspondence, see \cite{borisov2009pfaffian} we link $\tW$ to its (homological) projective dual $\tW^{\vee}$, whose hyperplane section is the above Campedelli surface $\tV$. This is the content of Proposition \ref{camp}.


\section{From Fano fourfolds to surfaces of general type}
Let $V_6 \cong \C^6$ and $V_7=\C^7$, where we fix (standard) bases for $V_6$ (resp. $V_7$) that we denote $v_1, \ldots, v_6$ and $v_1, \ldots, v_7$. We consider the Grassmannians of 2 planes in both $V_6$ and $V_7$, denoting them as Gr(2,6) and Gr(2,7) in their Pl\"ucker embedding.
Recall in general that the Grassmannians Gr(k,n) is a smooth subvariety of $\PP(\W^k \C^n) \cong \PP^{{n \choose k}-1}$ of dimension $k(n-k)$. On Gr(k,n) we have the standard (tautological) exact sequence
\begin{equation}
0 \to \mathcal{S} \to \of_G \otimes V \to \mathcal{Q} \to 0
\end{equation}
where $\mathcal{S},\mathcal{Q}$ denotes the tautological (resp. quotient) bundle of rank k (resp. n-k). One has $H^0(G, S^*)\cong V^*$ and $H^0(G, \mathcal{Q})\cong V$ and $\of_G(1) \cong det(\mathcal{S}^*)=det(\mathcal{Q})$, that gives the Pl\"ucker embedding. If we consider the zero locus inside the Grassmannian of a general global section of an homogeneous vector bundle, we get several interesting constructions. A complete classification of varieties of this type is still far away: partial results can be found for example in \cite{kuznetsov}, \cite{kuchle}, \cite{inoue2016complete} where we find a classification of Fano fourfolds of index 1 and Calabi-Yau 3-folds. In the following we focus on the geometry of some special varieties in the Grassmannians Gr(2,6) and Gr(2,7). \\
Let indeed $G_7=\Gr(2,7)$ and consider the following tower of linear sections \[ S_Z \subset W_Z \subset Z \subset G_7 \]
where each member of the tower is given by the zero scheme of a general global section of $\of_{G_7}(1)^{\oplus r}$, $r=6,7,8$. Equivalently, each of these is given by a general linear system $\Sigma \subset \W^2 V^*$ of the corresponding dimension, where we use $H^0(G_7, \of_{G_7}(1)) \cong \W^2 V_7^*$.
Since $\omega_{G_7} \cong \of_{G_7}(-7)$ by adjunction it is easy to see that $Z$ is a prime Fano fourfold of index $\iota_Z=1$, $W_Z$ is a Calabi-Yau threefold (already famous in literature for its application in Mirror Symmetry, see \cite{rodland2000pfaffian}) and $S_Z$ is a surface of general type with $\omega_{S_Z}=\of_{S_Z}(1)$. All of these three varieties shares $H^{\ddim}=42$. We can compute easily their Hodge numbers for example using Koszul complex and Bott's theorem.
\begin{lemma} \begin{itemize}
\item The only non-zero Hodge numbers for $Z$ are $h^{0,0}=h^{1,1}=h^{3,3}=h^{4,4}=1$ and $h^{1,3}=h^{3,1}=6, \ h^{2,2}=57$. Moreover $h^1(Z, T_Z)=42$;
\item The only non-zero Hodge numbers for $W_Z$ are $h^{0,0}=h^{1,1}=h^{2,2}=h^{3,3}=h^{3,0}=h^{0,3}=1$ and $h^{1,2}=h^{2,1}=50$. Moreover $h^1(Z, T_Z)=50$;
\item The only non-zero Hodge numbers for $S_Z$ are $h^{0,0}=h^{2,2}=1$ and $h^{2,0}=h^{0,2}=13, \ h^{1,1}=98$. Moreover $h^1(Z, T_Z)=56$;
\end{itemize}
\end{lemma}
%

\paragraph{Quotient bundle on Gr(2,6)}
Consider now the Grassmannian $G_6$=Gr(2,6) and $\mathcal{Q}(1)$ the rank four globally generated quotient bundle twisted by $\of_{G_6}(1)$. If $\lambda$ is a general global section in $H^0(G_6, \mathcal{Q}(1))$ its zero locus $Y_{\lambda}$ will be a smooth Fano fourfold, with $$K_{Y_{\lambda}} = (K_{G_6} \otimes det (\mathcal{Q}(1))|_{Y_{\lambda}}=\of_{Y_{\lambda}}(-6+5)=\of_{Y_{\lambda}}(-1).$$
We have a concrete description of the space of the global section of $\mathcal{Q}(1)$ given in \cite{manivel}. More precisely by Borel-Bott-Weil theorem we have $$ H^0(G_6, \mathcal{Q}(1))=\Ker (\lrcorner\colon (\W^2 V_6^*) \otimes V_6 \to V_6^*),$$ where $\lrcorner$ is the contraction operator. \\ In particular we have that any $\lambda \in  H^0(G, \mathcal{Q}(1))$ is an element in Hom($\W^2 V_6, V_6)$. For every $\lambda$ the corresponding $Y_{\lambda}$ will be  
\begin{equation} \label{lambda}Y_{\lambda} =\lbrace <a,b> \in \Gr(2,6) \ | \ \lambda(a,b) \in \ <a,b> \rbrace. 
\end{equation}
By taking two further hyperplane sections one gets even here a tower $$ S_Y \subset W_Y \subset Y_{\lambda} \subset G_6.$$
As one can check the invariants of the towers are the same once fixed the dimension: the reason for this coincidence has been explained by Manivel.
\begin{thm}[\cite{manivel}] $Z$ and $Y_{\lambda}$ are projectively equivalent.
\end{thm}
One has (see \cite{inoue2016complete} for the Calabi-Yau, and easy to see by hand as in the surface case) that $(W_Z,W_Y)$ and $(S_Z,S_Y)$ shares the same invariants as well.\\
We now start by defining our quotient construction, working both with the $Y$ and $Z$ model. We will focus on the cases of main interest for us, these being the fourfolds $Y,Z$ and the surfaces $S_Y, S_Z$, but of course everything can be adapted to the Calabi-Yau case $W_Y, W_Z$. Often, when computations will be identical, we will go into the details only for one model and just sketch the other. We will start defining in the following two different action of $D_7$, the dihedral group of order 14, on $V_6$ and $V_7$.

\subsection{Two representations of $D_7$}\label{eqn}
\subsubsection{$D_7$ acting on $V_6$}

Consider now the group $D_7$, acting on $\C^6$ via $$ \tau_6=\frac{1}{7}(1,2,3,4,5,6), \ \sigma_6=\begin{pmatrix}
0 & 0 & 0 & 0 & 0 & 1 \\ 
0 & 0 & 0 & 0& 1 &0 \\ 
0 & 0 & 0 & 1 & 0 & 0 \\ 
0 & 0 & 1 & 0 & 0 & 0 \\ 
0 & 1 & 0 & 0 & 0 & 0 \\ 
1 & 0 & 0 & 0 & 0 & 0
\end{pmatrix}, $$ that is $\tau_6 (v_i)=\epsilon^i v_i$ and $\sigma_6(v_i)=v_{7-i}$, $\epsilon$ being a primitive 7th root of 1. It is easy to see that $\tau_6$ and $\sigma_6$ satisfies the relations of the dihedral group, that is $$\tau_6^7=\sigma_6^2=\mathrm{Id}, \ \sigma_6 \tau_6^i=\tau_6^{7-i}\sigma_6.$$
The choice of this representation is motivated by some famous analogous constructions in the theory of surfaces of general type (for example the standard construction of a \emph{Godeaux surface} as a $\Z/5$ quotient of a smooth quintic surface).\\
The action of $\sigma_6$ pass to $\W^2V_6$ via the canonical associated representation $V_6^{\otimes 2}$, with the rule $$ \sigma_6 (v_i \wedge v_j)= v_{7-i} \wedge v_{7-j},$$ and as well to $\Gr(2,6)$, that we can identify as the set of totally decomposable 2-skew tensors in $\PP(\W^2 V_6)$.
With a little abuse of notation we will denote with $\rho_6$ both this representation of $\C^6$ and on $\W^2 \C^6.$ We consider now the subspace $\mathcal{Y}^{\rho_6}$ given by the $D_7$-invariant $Y_{\lambda}$ under the given representation $\rho_6$, that is $$\mathcal{Y}^{\rho_6}:= \lbrace Y_{\lambda} \ | \ \lambda(g \cdot [p]) \in Y_{\lambda}, \ g \in D_7, \ [p] \in Y_{\lambda} \rbrace,$$ where the $D_7$ action is computed according to $\rho_6$ and $\lambda$ is as in \ref{lambda}.\\
\begin{proposition} \label{inv}The family $\mathcal{Y}^{\rho_6}$  of $D_7$ invariant fourfolds of type $Y_\lambda$ has general member \[ \lambda= v_1 \otimes ( c_{2,6}v_2^* \wedge v_6^*+c_{3,5}v_3^* \wedge v_5^*)+v_2 \otimes (c_{3,6}v_3^* \wedge v_6^*+c_{4,5}v_4^* \wedge v_5^*)+v_3 \otimes (c_{1,2}v_1^* \wedge v_2^*+c_{4,6}v_4^* \wedge v_6^*)+\]\[+v_4 \otimes (c_{4,6}v_1^* \wedge v_3^*+c_{1,2}v_5^* \wedge v_6^*)+v_5 \otimes (c_{3,6}v_1^* \wedge v_4^*+c_{4,5}v_2^* \wedge v_3^*)+v_6 \otimes (c_{2,6}v_1^* \wedge v_5^*+c_{3,5}v_2^* \wedge v_4^*),\] where the $c_{i,j}$ are (sufficiently general) nonzero scalars.
\end{proposition}
\begin{proof}
Let us start writing a general element in $\W^2 V_6^* \otimes V_6$: this will be $$\lambda= \sum_{i,j,k} c_{i,j,k} v_i \otimes (v_j^* \wedge v_k^*).$$ If $a= \sum a_s v_s$ and $b=\sum b_sv_s$ are elements of $V$, we have that $$\lambda(a,b)= \sum_i v_i(\sum_{j,k} c_{i,j,k} p_{j,k}),$$ where $p_{j,k}= a_jb_k-b_ja_k$. The action of $D_7$ in terms of the generators can be expressed as $$\tau (\lambda(a,b))= \sum \xi^i v_i (\sum_{j,k} c_{i,j,k} p_{j,k}),$$ $$\sigma (\lambda(a,b))= \sum v_{7-i} (\sum_{j,k} c_{i,j,k} p_{j,k}).$$ On the other hand we have $$\lambda(\tau(a,b))= \sum  v_i (\sum_{j,k} \xi^{j+k} c_{i,j,k} p_{j,k}),$$ $$\lambda(\sigma(a,b))= \sum v_{7-i} (\sum_{j,k} c_{7-i,j,k} p_{7-j,7-k}).$$
This induces relations between $c_{i,j,k}$, namely \begin{enumerate}
\item $c_{i,j,k}=c_{7-i,7-j,7-k}$;
\item $c_{i,j,k}=0$ for $j+k\not \equiv i $ (mod 7)
\end{enumerate}
Expanding these conditions the statement follows.
\end{proof}

In order to get to a surface we need now to consider the zero set of a global section of $\mathcal{Q}(1) \oplus \of_G(1)^{\oplus 2}$: 
therefore we want to realize $S_{42}$ as $V(\lambda, h_1, h_2)$, whereas $h_1, h_2$ are two linear forms in Pl\"ucker coordinates. In order to preserve the surface we need to look for $D_7$ equivariant linear form as well: in particular, we need to work with the set $$ \mathcal{H}^{\rho_6}:= \lbrace (h_1, h_2) \in  (\W^2 V^*)^2 \ | \ (h_1, h_2) \textrm{ preserved by $D_7$ action } \rbrace$$
These will come by three copies of the trivial induced representation: in coordinates we have to check that, if $p= \sum l_{i,j} v_i \wedge v_j$ and $h_i= \sum h^i_{i,j} v_i^* \wedge v_j^*$, then if $p \in V(h_1, h_2)$, then $g \cdot p \in V(h_1, h_2)$ as well. It is easy to see that the action of $\tau$ and $\sigma$ combined implies that the two linear forms must be both of the form $$ h_i= h^i_{1,6} \ v_1^* \wedge v_6^*+h^i_{2,5} \ v_2^* \wedge v_5^*+ h^i_{3,4} \  v_3^* \wedge v_4^*.$$
Indeed we have
\begin{proposition} Any $D_7$ invariant surfaces $S_Z^{\rho_6}$ (with respect to the representation $\rho_6$) will be given by the triple $(\lambda_{\rho_6}, h_1, h_2)$, with $\lambda_{\rho_6}$ as in proposition \ref{inv} above, and $h_1, h_2$ in $\mathcal{H}^{\rho_6}$.
\end{proposition}
\subsubsection{From Gr(2,6) to Gr(2,7) and $D_7$ action}
In order to understand how the action of $D_7$ on $V_7$ works, we make explicit the identification between $Y$ and $Z$.
We use an alternative description given by Inoue-Ito-Miura, (cfr. \cite{inoue2016complete}, Proposition 4.1), that we briefly recall. Suppose $V$ is a linear space of dimension $n$, $\mathcal{E}$ a globally generated vector bundle on $Gr(k,V)$, $s$ an element in $H^0(\mathcal{E}) \otimes (\W^kV)^*$ and $\bar{s}$ its image in $H^0(\mathcal{E}(1))$. We denote by $P_{\bar{s}}$ the linear section of $Gr(k, V \oplus \C) \subset \PP(W \oplus \W^k W)$ given by the image of the map $$ \PP(\W^kV) \hookrightarrow \PP(H^0 (\mathcal{E}) \oplus \W^k V); \ \ [p] \to [\bar{s}(p), p],$$ where $W=V \oplus \C$.\\
One has that $\bar{s}$ is general if and only if $P_{\bar{s}}$ is, and $V(\bar{s})$ and $V(P_{\bar{s}})$ are projectively equivalent. This is exactly our case with $\mathcal{E}=\mathcal{Q}$ and $\bar{s}=\lambda$. Therefore computing the image of the map above one has that $Z=V(P_{\bar{s}}) \subset Gr(2,7)$ is defined by the following 6 equations\begin{equation}\label{maximal}Z=V(x_{1,7}-c_{2,6}x_{2,6}-c_{3,5}x_{3,5}, x_{1,6}-c_{3,6}x_{2,5}-c_{4,5}x_{3,4}, x_{1,5}-c_{4,6}x_{2,4}-c_{1,2}x_{6,7},\end{equation} 
$$ \ \ \ x_{1,4}-c_{1,2}x_{2,3}-c_{4,6}x_{5,7},x_{1,3}-c_{3,6}x_{4,7}-c_{4,5}x_{5,6},x_{1,2}-c_{2,6}x_{3,7}-c_{3,5}x_{4,6}).$$

This suggests indeed how the $D_7$ action on $V_7$ should work. In particular we define $\tau_7=\frac{1}{7}(0,1,2,3,4,5,6)$ and $\sigma_7(v_i)=v_{9-i}$. This passes to $\W^2V_7$ via $$\tau_7(v_i \wedge v_j)= \epsilon^{i+j-2} v_i \wedge v_j$$ and $$\sigma_7(v_i \wedge v_j)=v_{9-j}\wedge v_{9-i}.$$ We denote this representation by $\rho_7$.  With computations totally similar to the case $n=6$, one find  after rescaling the first coefficient of every equation that the maximal invariant family is indeed what we already found above
\begin{lemma} The maximal family $Z_{\rho_7}$ of invariant fourfold with the action above defined is the complete intersection defined by the equation in \ref{maximal}
Notice that in any of the above equations the sum $i+j \equiv k \ (mod \ 7)$ is constant ($k=1, \ldots, 7$,  $k \neq 2$).\\
Similary the maximal family $S_{\rho_7}$ is obtained by adding two copies coming from the trivial representations, that is two (linearly independent) hyperplanes in the coordinates $x_{3,6}, x_{4,5}, x_{2,7}$.
\end{lemma}
We want to rewrite the generic member of the above family of surfaces in a much more neat style. Recall that taking the 4-Pfaffians of a generic skew $7 \times 7$ matrix of linear forms yields the Pl\"ucker equations of the Grassmannian $\Gr(2,7)$. We can write our invariant family in the format\begin{equation*}
M= \left(\begin{smallmatrix}
 & \mu_1 x_{3,7}+\mu_2 x_{4,6} & \mu_3 x_{4,7}+\mu_4x_{5,6} & \mu_5 x_{2,3}+\mu_6x_{5,7} & \mu_6x_{2,4}+\mu_5x_{6,7} & \mu_3x_{2,5}+\mu_4x_{3,4} & \mu_1 x_{2,6}+\mu_2 x_{3,5} \\ 
 &  & x_{2,3} & x_{2,4} & x_{2,5} & x_{2,6} & \epsilon_1 x_{4,5} \\ 
 &  &  & x_{3,4} & x_{3,5} & \epsilon_2 x_{4,5} & x_{3,7} \\ 
 &  &  &  & x_{4,5} & x_{4,6} & x_{4,7} \\ 
&  & &  &  & x_{5,6} & x_{5,7} \\ 
 &  &  &  &  &  & x_{6,7} \\ 
 &  &  &  &  &  & 
\end{smallmatrix}\right)
\end{equation*}
where of $S_Z \subset \PP^{12}$ is 
\begin{equation}\label{format}S_Z=V(\Pf(4,M)).
\end{equation} 
The parameters $\epsilon_1$ and $\epsilon_2$ come from the solution of the system of two equations in the $x_{3,6}, x_{4,5}, x_{2,7}$. Equation for the generic Calabi-Yau and fourfold can be easily accessed plugging back in $x_{3,6},x_{2,7}$.

\subsection{Simultaneous smoothness and fixed locus of the action}
Before taking the quotient, we need to address the question of the smoothness of our specific fourfolds $Y_{\lambda}$ and $Z$. As said before, by Inoue-Ito-Miura it suffices to check this for the $Z$-model (since the smoothness of $Z$ implies the generality of $\lambda$, and therefore the smoothness of $Y_{\lambda}$).\\

\begin{lemma} The general surface $S_Z$ constructed above is smooth.
\end{lemma}
\begin{proof}
The smoothness of $Z$ can be checked in several ways, for example by computing the infinitesimal deformation module of the affine cone of the general member or with a computation in local coordinates. We require our coefficients to be sufficiently general, for example all distinct numbers. On the other hand it is easy to to produce singular example with some special choice of coefficients. For example by picking all $\mu_i=1$ one gets a nodal surface. We propose  here an alternative computer-free method coming from the theory of \emph{exterior differential systems} (see \cite{bryant2013exterior}). We use a sufficient criterion for a point in a linear section of a Grassmannian of planes to be smooth.\\
In general, let $V$ a vector space, $\Sigma \subset  \W^2(V^*)$ a linear subspace and $Z_{\Sigma}$ the corresponding subvariety of the Grassmannian. For any $w \in V$, consider the vector space $H(w)$ defined as $$H(w) = \lbrace a \in V \ | \ \Omega(a,w)=0, \ \textrm{for all } \Omega \in \Sigma \rbrace. $$
We say that $w$ is $\Sigma$-regular if the dimension of $H(w)$ is minimal among all $w \in V$ and that a 2-plane $P \in Z_{\Sigma}$ is $\Sigma$-ordinary if $P$ contains a $\Sigma$-regular vector. The relevant result is that any ordinary plane is actually a smooth point of $Z_{\Sigma}$.\\
Let us now apply this method to our case. Let us do first the surface $S_Z$. Fix a $w=\sum w_i v_i$: $H(w)$ is then exactly the space of point $u$ in $\C^7$ that satisfies the system of equations \ref{maximal}, with two more in the coordinates $x_{3,6}, x_{4,5}, x_{2,7}$. This amounts to solve the linear system  $$M \cdot U=0,$$ where $$M=(\mu_k w_i)_{k,i}, \ \ U=(u_1, \ldots, u_7)^T.$$ One checks that for general $\mu_k$ and $w_i$ the matrix has maximal rank (that is, the dimension of $H(w)$ is constantly zero for general choices) and that any plane $P$ in $S_Z$ contains a general $w$. 
\end{proof}
By applying the same method one checks
\begin{lemma} The general fourfolds $Z$ and $Y_{\lambda}$ constructed above are smooth.
\end{lemma}
\begin{proof}
The above method works  for every $P \in Z$, except $p_{3,6}, p_{4,5}, p_{2,7}$ (recall that these three points do not belong to $S_Z$). In fact one checks that for any $w$ in these three planes the corresponding $H(w)$ has dimension two, instead of the expected one. A local computation on the Grassmannian (using for example the chart $p_{1,2}=1$) shows that even these three points are smooth points of $Z$.
\end{proof}

\subsubsection{Fixed locus of the action}
\paragraph{The $Z$-model} 
Once estabilished the smoothness of the fourfolds $Y_{\lambda}$ and $Z$ (and the same for the surface $S_Z$) of the maximal $D_7$ invariant families, we have to compute the fixed locus for the elements of the group. The $Y$-model is identical, therefore we will just sketch the computations. We include in bibliography a link, \cite{documentation} to the M2 and MAGMA codes that we have used for computations.\\
In particular we find that all the order 7 elements of the dihedral group have no fixed points on the surface and the Calabi-Yau, so the subgroup $\Z/7 \triangleleft D_7$ acts freely. Each of the 7 conjugate involutions fixes a conic and 10 isolated points (different for each involution). In particular we have the following
\begin{lemma} \label{fixloc}The fixed locus for the action of the group $D_7$ is 
\begin{itemize}
\item on the surface $S_Z$ and on the Calabi-Yau $W_Z$ it is the reducible union  $\bigsqcup_{i=1}^7 C_i$, where each $C_i$ is the union of a plane conic and 10 extra (disjoint) isolated points. Moreover all $C_i$ are conjugates under the normal subgroup $\Z/7$; 
\item on the fourfold $Z$ it has 3 extra fixed points .
\end{itemize}
\end{lemma}
\begin{proof}

 Consider first the cyclic action of $\tau_7$ as $\frac{1}{7}(0,1,2,3,4,5,6)$ when induced on the exterior algebra. In particular it sends $$\tau_7(\sum \lambda_{i,j} v_i \wedge v_j) \to \sum \epsilon^{i+j-2} \lambda_{i,j} v_i \wedge v_j.$$
The fixed locus on $\PP^{20}$ is the union of seven $\PP^2$, each one with coordinates $\lbrace x_{i,j} \rbrace_{i+j \equiv k \ (mod \ 7)}$. A computer check shows that the cyclic fixed locus lies away from $W$ and $S$. We give in the following a computer-free proof.\\
We have two types of fixed points: the coordinate points $p_{i,j}$ and any other of the form \begin{equation}\label{free}
 \lambda_{i,j} v_i \wedge v_j, \textrm{ with }i+j \equiv\textrm{const} \ (mod \ 7).
 \end{equation}
It is easy to check that no coordinate points $p_{i,j}$ belongs to $Z$ except $p_{3,6}, p_{2,7}, p_{4,5}$ (and they do not belong to $S_Z$). We claim now that the locus in \ref{free} does not intersect the Grassmannian Gr(2,7). To see this recall that the Pl\"ucker equations for the Grassmannian Gr(2,7) are obtained by picking the 4-Pfaffians of the 7x7 skew-symmetric matrix $M=(m_{i,j})$ with $m_{i,j}=\pm x_{i,j}$ if $i< j$ (resp. $i>j$) and $m_{i,j}=0$ if $i=j$.
 By looking at the action of $\tau_7$ any point (of non-coordinate type) of the form \ref{free} can have either two or three non-zero coordinates, with the sum of the indeces being constant mod.7. Call these $(i,j),(k,h), (r,s)$. Substituting in the Pl\"ucker in both case we will have either a surviving (say) $x_{i,j}x_{r,s}=0$ or all three possibilities. In both cases, this implies that none of these points belongs to the Grassmannian.\\
What happens now with of $\sigma_7$? Recall the construction in \ref{format}. The fixed locus of the involution on the ambient $\PP^{12}$ is given by the disjoint union of $\PP^+ \sqcup \PP^-$, with $$\PP^+=V(x_{3,7}-x_{2,6}, \ldots , x_{2,4}-x_{5,7})$$ and $$\PP^-=V(x_{3,7}-x_{2,6},\ldots, x_{2,4}-x_{5,7}, x_{4,5}).$$ Intersecting with the 35 Pfaffians this gives us the union of $C_1 \sqcup C_2$ with $C_1$ being 10 points and $C_2$ a smooth plane conic. All the other six (conjugate) involutions yields the same type of fixed locus. The result follows.\end{proof}

\paragraph{The $Y$-model}
Computations here are identical to the $Z$ model, and yields the same results. One has just to verify that $\tau_6$ yields (on the fourfold $Y_{\lambda}$) the points $p_{1,6}$, $p_{2,5}, p_{3,4}$ whereas the fixed locus of the  involutive part comes from the intersection with the zero set of the equations $\lbrace x_{i,j}\pm  x_{7-j,7-i} \rbrace$, and the same for the other conjugate involutions. We just want to remark that even $Y$ admits a concrete description in terms of equations in $\PP^{14}$. Recall from \ref{lambda} the description of $Y_{\lambda}$ as \begin{equation*} Y_{\lambda} =\lbrace <a,b> \in \Gr(2,6) \ | \ \lambda(a,b) \in \ <a,b> \rbrace. 
\end{equation*}
If $a_1, \ldots, a_6$ and $b_1,\ldots, b_6$ denotes the coordinates of $a,b$ with respect to the standard basis fixed, and if we call $p_{i,j}=a_ib_j-b_ia_j$ the Pl\"ucker coordinates in $\W^2V_6$ the condition above translates in matrix form as

$$\mathrm{ rk} \begin{pmatrix}
a_1 & a_2 & a_3 & a_4 & a_5 & a_6 \\ 
b_1 & b_2 & b_3 & b_4 & b_5 & b_6 \\ 
p_{2,6}+p_{3,5} & p_{3,6}+p_{4,5} & p_{1,2}+p_{4,6} & p_{1,3}+p_{5,6} & p_{1,4}+p_{2,3} & p_{1,5}+p_{2,4}\end{pmatrix} =2.$$
Expanding the determinant in Laplace forms one gets quadratic equation in the dual of the Pl\"ucker coordinates, getting in this way $Y_{\lambda}$ as explicit subvariety of $\PP^{14}$.

\subsection{Quotient Calabi-Yau threefold and surface of general type with an involution}
The analysis in the previous paragraph shows how the fixed locus of the dihedral group $D_7$ depends only on the seven conjugate involutions. In particular the normal subgroups $\Z/7 \triangleleft D_7$ yields a free action on each member of the invariant family, both in the Calabi-Yau and in the surface case. We can then take the quotient for such subgroup and produce new families of varieties in dimension (respecively) 2 and 3. Since we can perform the construction in both $Y$ and $Z$ model, we will simply write $W$ and $S$.
\begin{thm}\label{cyquotient} Let $W$ a linear section of the Grassmannian Gr(2,7) constructed as above. Then $W$ admits a free $\Z/7$ action. In particular the quotient $\pi: W \to \tW$ yields a smooth Calabi-Yau threefold.
\end{thm}
\begin{proof}
Follows from description in lemma \ref{fixloc}, where an explicit description of the fixed locus of the dihedral group on $W$ is given.
\end{proof}
\begin{corollary} The Calabi-Yau $\tW$ has Euler characteristic $\chi(\tW)=-14$. In particular the Hodge diamond of $\tW$ is

\[ \begin{matrix}
1 && 8 && 8 &&1& \\
&0 &&1&&0&\\
&&0&&0&&\\
&&&1 &&&
\end{matrix}\]
\end{corollary}
\begin{corollary} \label{surface} Let $\tS$ the surface of general type obtained by intersecting $\tW$ with a $\Z/7$-invariant hyperplane section. Then $p_g(\tS)=1, q(\tS)=0$, $K^2_{\tS}=6$, $\pi_1= \Z/7$. In particular its Hodge diamond is
\[ \begin{matrix}
&1 &&14&&1&\\
&&0&&0&&\\
&&&1 &&&
\end{matrix}\]
\end{corollary}
The minimality of the above surface $\tS$ follows from the fact that the generic member of the family of $S$ has Picard rank $\rho=1$.\\
As one can see from (\ref{format}) the surface construction depends by 8 parameters. This is indeed the expected number of moduli $M$. In fact we have
 $$h^1(\tS, T_{\tS}) \geq M \geq h^1(\tS, T_{\tS})-h^2(\tS, T_{\tS})=10\chi(\of_{\tS})-2K_{\tS}^2=8.$$
 We conjecture that our 8-parameter family is indeed an irreducible component of the moduli space of surfaces. We
 notice moreover that the whole family is unobstructed. Since the covering map $\pi: S \to \tS$ is finite, we have $H^i(T_S)\cong H^i(\pi_* T_S)$. To show that this family of $\tS$ is unobstructed it suffices to observe the following lemma.
\begin{lemma} Let $S_Z$ be a codimension 8 (linear) complete intersection in the Grassmannian Gr(2,7). Then $H^2(S_Z, T_{S_Z})=0$.
\end{lemma}
\begin{proof}
To $S_Z$ is associated the standard tangent sequence $$0 \to T_{S_Z} \to T_{\Gr}|_{S_Z} \to (\of_{S_Z}(1))^{\oplus 8} \to 0.$$ Passing in cohomology we get $$ \ldots \to 0 \to (H^1(\of_{S_Z}(1))^{\oplus 8} \to H^2(S_Z, T_{S_Z}) \to H^2(S_Z, T_{\Gr}|_{S_Z}) \to \ldots$$
Since $(H^1(\of_{S_Z}(1))^{\oplus 8}=0$, the claim will be proved if $H^2(S_Z, T_{\Gr}|_{S_Z})=0$. To prove this first notice that $T_{\Gr} \cong \Omega^9(7)_{\Gr}$. We then use the Koszul complex for a complete intersection in a Grassmannian after tensoring with $T_{\Gr}$. In particular we have $$  \cdots \to (T_G(-1))^8 \to T_{\Gr} \to T_{\Gr}|_{S_Z} \to 0.$$ Splitting in short exact sequences, we have that we will have vanishing of $H^2(S_Z, T_{\Gr}|_{S_Z})$ if both $H^2(T_{\Gr})$ and $H^3((T_{\Gr}(-1))$ does the same. But these are isomorphics to (resp.) $H^2(\Gr, \Omega^9(7))$ and $H^3(\Gr,\Omega^9(6))$, and these vanishing are automatic for the Grassmannian Gr(2,7) (see \cite{peternell}, lemma 0.1).
\end{proof}
The surface $\tS$ that we have constructed is not a Todorov surface, neither one constructed by Park-Park-Shin in \cite{park}. Indeed the latter are simply connected, while for the former one can check that $\tS$ is not contained in one of the 11 non empty irreducible families listed in \cite[pg 335]{morrison}.\\
The surface $\tS$ comes with a involution $\sigma: \tS \to \tS$. The fixed locus of the involution $\sigma$ consists in one smooth plane conic $C$ and 10 isolated points. We can take the quotient $\sigma: \tS \to \tS/\sigma=: \Sigma$. By adjunction formula $K_{\tS}= \sigma^*(K_{\Sigma})+C$: therefore
\begin{equation}\label{proj}K^2_{\Sigma}= \frac{K^2_{\tS}+C^2-2K_{\tS}C}{2}.\end{equation}
Moreover the adjunction formula for curves on a surface says $K_{\tS}C+C^2+2\chi(\of_C)=0$. \\
Similar formulae relates $\Sigma$ and $S$, the surface of degree 42, where the group acting is the full $D_{14}$ and the fixed locus is given by one conic and 10 isolated points for each of the conjugate involutions. From these and and a computation using MAGMA one gets $C^2=-4$. Therefore by adjunction $K_{\tS}C=2$ and by (\ref{proj}) $K^2_{\Sigma}=-1$. In particular, the Kodaira dimension $k(\Sigma) \leq 0$. The surface $\Sigma$ is of course non minimal, and it has $10\times A_1$ singularities. Using \cite[Lemma 3]{rolleau} we compute $e_{\mathrm{top}}(\Sigma)=\frac{18+2+10}{2}=15$. Denoting with $\widehat{\Sigma}$ the minimal resolution for $\Sigma$, Noether formula and regularity for $\tS$ imply that $e_{\mathrm{top}}(\widehat{\Sigma})=25$ and $p_g(\widehat{\Sigma})=1, q(\widehat{\Sigma})=0$. Therefore $\Sigma$ is a K3 surface with 10 nodes, blown up a single time.
\\
Our construction could be linked to an example of a surface of general type with $p_g=q=0, K^2=3$, and a fundamental group of order 14.  For this one would need a fix-point-free involution on our surface. As we have seen, the involution $\sigma$ has indeed a fixed locus, making impossible to extend this construction any further.

 \section{Appendix A: Further group invariances on the $Z$-model}

 \subsection{Frobenius group of order 21}
The dihedral group $D_7$ is not the biggest group under which the family of surfaces is invariant. To see this, let us rewrite $S \subset \PP^{12}$ in a way inspired by Reid's construction of the $\Z/7$ Campedelli surface (full details later on). Namely, pick coordinates $x_1, \ldots, x_6, y_1, \ldots, y_6, z$ and define $S=V(\Pf(4,M))$ with \[M=\begin{pmatrix}
 0 & x_1+y_1 & x_3+y_3 &x_2+y_2 & x_6+y_6 & x_4+y_4 & x_5+y_5 \\ 
   & 0 &x_4 & \lambda_3y_3 & z & -\lambda_5y_5 & -x_6 \\ 
   &   & 0 &x_5 & \lambda_2y_2 &z& -\lambda_1y_1 \\ 
   &   &   & 0 & x_1 & \lambda_6y_6 &z\\ 
   & -\textrm{sym} &   &   & 0 &x_3 & \lambda_4 y_4 \\ 
   &   &   &   &   & 0 & x_2\\ 
   &   &   &   &   &   & 0
 \end{pmatrix} 
 \]
 Denote by $a$ the cyclic generator sending $x_i \mapsto \varepsilon^i x_i$, $y_i \mapsto \varepsilon^i y_i$, $z \mapsto z$ and $b$ the generator sending $x_i \mapsto x_{2i}$, $y_i \mapsto y_{2i}$, $z \mapsto z$. This corresponds to the cycle $(2, 4,6)(3,5,7)$. Denote by $F_{21}$ the group (of order 21) generated by $a,b$. One checks that $ab=b^2a$. Therefore by the classification of small groups, $F_{21}$ is isomorphic to the Frobenius group of order 21,  which can be represented as the subgroup of $S_7$ generated by $(2,3,5)(4,7,6)$ and $(1,2,3,4 ,5, 6, 7)$, and is the Galois group of $x^7 - 14x^5 + 56x^3 -56x + 22$ over the rationals. The fixed locus is given by imposing $x_1=\rho^i x_2=\rho^{2i}x_4$ (and so on for the other coordinates), where $\rho$ is a third root of unity. It consists of 3 points.\\
We point out that the family is invariant under the group $G_{42}$ of order 42 generated by $a$ and $b'$, with $b': x_i \mapsto x_{3i}$. This construction can be adapted in a straighforward way from the one already given in \cite{miles}.

\subsection{Another $D_7$ action}
The dihedral action we defined is not the only one that can be constructed on the Grassmannian. Indeed we may specify a point in the Grassmannian $\Gr(k,n)$ as a $k \times n$ matrix. The symmetric group $S_{n}$ then acts permutating the columns. Thus the dihedral subgroup $D_{n}$ of $S_n$
generated by the $n$-cycle $\alpha: (1,2,\ldots, n)$ and the longest element in the group $w_0$. The latter, in the case of the symmetric group, corresponds to the permutation $i \mapsto n+1-i$.\\ 
In this case the involutions corresponds to our original one, while the order seven element comes from the discussion in the above subsection. In more concrete terms, define $S \subset \PP^{12}$ be the zero set of the linear equation $$ H= \sum_{i=1}^7 \lambda x_i + \sum_{i=1}^7 \mu y_i$$ and the 4-Pfaffians of the matrix

\[M=\begin{pmatrix}
 0 & \lambda x_6+\mu y_6 & \lambda x_2 &\mu x_5 & \mu y_1 & \lambda x_4  & \lambda x_7+\mu y_7 \\ 
   & 0 &\lambda x_5+\mu y_5 & \lambda x_1 &\mu y_4 & \mu y_7 & \lambda x_3 \\ 
   &   & 0 &\lambda x_4+\mu y_4 & \lambda x_7 &\mu y_3& \mu y_6 \\ 
   &   &   & 0 & \lambda x_3+ \mu y_3 & \lambda x_6 &\mu y_2\\ 
   & -\textrm{sym} &   &   & 0 &\lambda x_2+\mu y_2 & \lambda x_5 \\ 
   &   &   &   &   & 0 & \lambda x_1+\mu y_1\\ 
   &   &   &   &   &   & 0
 \end{pmatrix} 
 \]
 The action of the 7-cycle $\alpha$ sends $x_1 \mapsto x_2 \mapsto \ldots \mapsto 7 \mapsto 1$ for both $x_i$ and $y_i$, while $w_0$ sends $x_1 \mapsto x_6$, $x_2 \mapsto x_5$ and $x_3 \mapsto x_4$, keeping $x_7$ fixed (and similar for $y_i$).
 The surface defined above is clearly invariant under this new dihedral action: however, if we compute the fixed locus we got the same answer of the old model (that is, a smooth conic and 10 isolated points).

\section{Appendix B: Invariant surface family in the Grassmannian Gr(3,6)}

The Grassmannian $\Gr(3,6)$ shares many numerical similiarities with the Grassmannian $\Gr(2,7)$. First of all notice how the Pl\"ucker spaces have very similar dimensions (19 and 20, respectively).
 Moreover the dimension of the Grassmannian $\Gr(3,6)$ is 9, and defined exactly by 35 Pl\"ucker quadrics. Both Grassmannians have degree equals to 42. Of course $\Gr(3,6)$ is not an hyperplane section of $\Gr(2,7)$, nevertheless a further (and even more relevant) similarity comes from their Hilbert-Poincar\'e Series. One has in fact $$ \HP(\Gr(3,6)= \frac{P(t)}{(1-t)^{19}}; \ \HP(\Gr(2,7)=\frac{P(t)}{(1-t)^{20}},$$ with the same Hilbert numerator $P(t)$.\\
Consider now a eight-codimensional linear section of the Grassmannian $\Gr(2,7)$ and a seven-codimensional linear sections of the Grassmannian $\Gr(3,6)$. The first one is the already considered $S_{42}$, and let us call $T$ the second one. Of course both $S$ and $T$ by Lefschetz theorem are regular surface, of degree 42 and by adjunction their canonical class $\omega \cong \of(1)$. Moreover, since the Hilbert numerators are the same for both Grassmannians, they have the same numerical invariants. The idea is try to replicate the $D_7$ construction on the $\Gr(3,6)$ model. Note that the same construction cannot extend to the Calabi-Yau case in dimension 3. Indeed a 6-codimensional linear section in $\Gr(3,6)$ has Euler characteristic -96, ruling out even the possibility of any fix-point-free action of a group with order divisible by seven. \\
As before, we have to build up a $D_7$ action on $V_6$ and later on extend to the Grassmannian. Let us define this action by sending $x_i \mapsto \varepsilon^i x_i, \ \ x_i \mapsto x_{6-i}.$ This action extends to $\W^3 V_6$ in the obvious way, with $x_{i,j,k} \mapsto \varepsilon^{i+j+k} x_{i,j,k}, \ \ x_{i,j,k} \mapsto - x_{6-i,6-j,6-k}.$
It is easy to see that the Grassmannian Gr(3,6) is preserved under this action. The problem reduces then to find an invariant $\PP^{12}$, as in the previous cases. Observe now that any $\Z/7$ eigenvalue different from zero can be obtained in three distinct way as sum mod 7 of strictly increasing natural numbers between 1 and 6. For example $1\equiv 1+2+5 \equiv  1+3+4 \equiv 4+5+6$ and so on. Zero behaves differently, since we have only $0 \equiv 1+2+4 \equiv 3+5+6$.  We can therefore build up equations for $T$ by picking $ T=V( \ldots, \sum_{i+j+k \equiv c} \alpha_{i,j,k} x_{i,j,k}, \ldots).$ Choosing the $\alpha_{i,j,k}=\alpha_{6-i,6-j,6-k}$ we immediately obtain not only the $\Z/7$ invariance but the full $D_7$ as well. \\By doing computations totally similar to the one in the Gr(2,7) case one shows that the $\Z/7$ part of the action is free. Each conjugate involution fixes an elliptic curve $E$ of degree 6 and 6 distinct points. In particular by adjunction formula $K_T \cdot E=6 \Rightarrow E^2=-6$ and $K^2_{T/\sigma}=\frac{6+E^2-2K_T \cdot E}{2}= -6.$   
We point out that we have not been able to check the smoothness of $T$ for generic coefficients without appealing to a \emph{tour-de-force} in computational algebra. We can state anyway the following proposition.
\begin{proposition} \label{g36}Let $T$ a smooth surface constructed as above. The quotient $T/\Z/7$ is a smooth surface of general type with $p_g=1, q=0, K^2=6$, together with an involution $\sigma$.
\end{proposition}
One very interesting question would be to relate the surface constructed from $\Gr(3,6)$ with the one constructed from $\Gr(2,7)$. However, we have not been able to do so yet.\\
Another interesting feature of the $\Gr(3,6)$ model is that it seems to possible to construct on it an involution $\sigma$ (not induced from $V_6$) such that generates together with the above $\Z/7$ a cyclic group of order 14, with the quotient $T/\Z/14$ having $p_g=0, \ K^2=3$. This will be part of a forthcoming work of the author together with Borisov.

 \section{Appendix C: link with the Pfaffian-Grassmannian equivalence and the Reid $\Z/7$-Campedelli surface}
 
Our construction is closely related to another famous minimal surface of general type, the $\Z/7$ Campedelli-Reid surface from \cite{miles} . This goes via another well known geometric construction, the \emph{Pfaffian-Calabi Yau correspondence}, considered by many authors in \cite{rodland2000pfaffian}, \cite{borisov2009pfaffian}. \\Before making everything explicit, we recall the two main ingredients of the construction.

\subsection{The Pfaffian-Grassmannian equivalence}
We want to describe now another Calabi-Yau $W^{\vee}$ related to our $W$.  We will follow the description of Borisov-Caldararu in \cite{borisov2009pfaffian}. Let fix $V$ as the vector space of dimension 7. If $W \subset \Gr(2,7) \subset \PP(\W^2 V) \cong \PP^{20}$, take the
 dual projective space 
$\PP^* = \PP(\wedge^2 V^*) $
as the projectivization of the space of two-forms on $V$.  The
Pfaffian locus 
$\Pf \subset \PP^*$
is defined as the projectivization of  forms of rank $\leq 4$ on $V$ (that is, degenerate).  Equations for $\Pf$ can be
obtained by taking the maximal Pfaffians of a
skew-symmetric $7 \times 7$ matrix of linear forms on $V$. Note that this yields cubic equation.
The Pfaffian $\Pf$ is a singular subvariety of $\PP^*$ of dimension $17$, with a point
$\omega\in \Pf$ will be singular precisely when the rank of $\omega$
is two. 

Consider a linear subspace of dimension seven
$H^{\vee}\subset \wedge^2 V^*,$
and by abuse of notation $H^{\vee} $ will denote its image in $\PP^*$ as well.  Let $W^{\vee}$ be the
intersection of $H^{\vee} $ with $\Pf$.
Dually let 
$H = \Ann(H^{\vee}) \subset \wedge^2 V $
be the $14$-dimensional annihilator of $H^{\vee}$; and $W$ be the intersection of
$H$ and $\Gr$.  From the construction is evident that $W^{\vee}$ is the projective dual of the $W$ we started from. $W$ and $W^{\vee}$ are not even birational (indeed they have different degrees and $\rho=1$), but enjoy deep similarities. A famous results in \cite{borisov2009pfaffian} establish as an example their derived equivalence.
%
%
%
%
\subsection{The Campedelli-Reid $\Z/7$ surface}
Recall the construction of the $\Z/7$ Campedelli-Reid surface from \cite{miles}.\\
The aim is constructing a canonically embedded and projectively Cohen-Macaulay surface of general type $V \subset \PP^5$ with $p_g=6, \ K^2=14$. These hypotheses implies that the coordinate ring 
is Gorenstein and of codimension 3. In particular, by the famous structure theorem of Buchsbaum-Eisenbud, the ideal of relation can be written as submaximal Pfaffians of a $7 \times 7$  skew matrix. One shows that if the entries $l_{ij}$
of M are sufficiently general then $V : (\Pf_i = 0)$ has the stated properties.\\
Our purpose now is to construct a free action of the group $\Z/7$ on $V$. The general $V$ will not be $\Z/7$-invariant, but we can still get an invariant subfamily by choosing $M$ carefully.  This can be done by arranging the entries of the $7 \times 7$ matrix in a clever ad-hoc way, see \cite{miles}.
%
Moreover for sufficiently general values of the parameters the surface $V=V(\Pf_0=\ldots=\Pf_6=0)$ is smooth, and therefore one has
\begin{thm}[\cite{miles}] \label{reidcamp} Pick $M$ as above, and $V \subset \PP^5$ the corresponding surface. The quotient $\tV=V/ \Z/7$ is a smooth surface of general type with $p_g=q=0$, $K^2=2$, that is a Campedelli surface.
\end{thm}

\subsection{From our surface to the Campedelli-Reid}
Consider now the quotient Calabi-Yau $\tW$ constructed in \ref{cyquotient}, and let now $\tW^{\vee}$ the dual variety to $\tW$ as constructed above. Denote by $\Pf(V)$ the Pfaffian variety in $\PP^{20}$.  
\begin{proposition} \label{camp} $\tW^{\vee}$ is the extension to a Calabi-Yau threefold of the Campedelli-Reid $\Z/7$ surface. In particular if $H_7$ is a $\Z/7$-invariant hyperplane section one has $\tW^{\vee} \cap H_7 = \tV$, with $\tV$ as in the section above.
\end{proposition}
\begin{proof} Recall that the equations of a $\Z/7$-invariant $W$ are the one listed in \ref{maximal}, to which we have to add one further linear equation in the variables $x_{2,7}, x_{3,6}, x_{4,5}$ (corresponding to the 0-eigenspace). In particular such seven equation will form a seven-dimenisonal linear subspace $\PP(\Lambda) \subset \PP(\W^2V^*)$. Equations for the dual variety $W^{\vee}$ can be then obtained by considering $\PP(\Lambda^{\perp})$. Note that this gives us 14-codimensional linear section of the Pfaffian variety, grouped by their eigenvalue with respect of the $\Z/7$ action. For example we will have $$W^{\vee}= V(x_{1,2}- \mu_1 x_{3,7}, x_{3,7}-\mu_2x_{4,6}, \ldots) \subset \Pf(V),$$ and so on according to the same rule. Therefore we can project down to the $\PP^6$ with coordinates $x_{1,2}, \ldots, x_{1,6}, x_{2,7}$, where we chose one representative for any eigenspace. The dual variety obtained $W^{\vee}$ will be smooth if only if $W$ is so by \cite{borisov2009pfaffian}. Anyway, since the codimension is small, we can directly check the smoothness of $W^{\vee}$ by any computer algebra system. One can see directly that the equations for $W^{\vee}$ can be arranged in Pfaffian format inside the matrix

 \[M=\begin{pmatrix}
 0 & x_{1,2} & x_{1,4} & x_{1,3} & x_{1,7} & x_{1,5} & x_{1,6} \\ 
   & 0 & x_{1,5} & \lambda_3x_{1,4} & x_{2,7} & -\lambda_5x_{1,6} & -x_{1,7} \\ 
   &   & 0 & x_{1,7} & \lambda_2x_{1,4} & x_{2,7} & -\lambda_1x_{1,3} \\ 
   &   &   & 0 & x_{1,3} & \lambda_6x_{1,7} & x_{2,7} \\ 
   & -\textrm{sym} &   &   & 0 & x_{1,5} & \lambda_4 x_{1,5} \\ 
   &   &   &   &   & 0 & x_{1,3}\\ 
   &   &   &   &   &   & 0
 \end{pmatrix} 
 \]
 with appropriate parameters. By the same argument of \cite{miles} one has that the Pfaffians are $\Z/7$ invariant, and therefore realize the quotient $\tW^{\vee}$. Moreover notice that by picking one further $x_{2,7}=0$ one gets down exactly to the equations of the $\Z/7$ Campedelli-Reid surface described in \ref{reidcamp}.
\end{proof}
\subsection*{Acknowledgments}
This work is part of the PhD thesis of the author. I wish to thank my supervisor Prof. Miles Reid for introducing me to the subject, Christian B\"ohning, Alessio Corti and Rita Pardini for useful comments and discussions on the topic.
I wish to thank the anonymous referees for their comments and insights on this work, which led to an improvement of the work itself.
The author has been supported by MIUR-project FIRB 2012 "Moduli spaces and their applications" and member of the INDAM-GNSAGA.

\end{document}